\documentclass[11pt]{amsart}
\usepackage{amsmath,amssymb,amsthm}
\usepackage{hyperref}
\usepackage{mathrsfs}
\usepackage{graphicx}
\usepackage{tikz}
\usepackage{enumerate,enumitem}
\usepackage{float}

\theoremstyle{plain}
\newtheorem{theorem}{Theorem}[section]
\newtheorem{lemma}[theorem]{Lemma}

\newtheorem{proposition}[theorem]{Proposition}
\newtheorem{corollary}[theorem]{Corollary}

\theoremstyle{definition}

\newtheorem{remark}[theorem]{Remark}
\newtheorem{example}[theorem]{Example}

\newtheorem*{namedthm*}{\thistheoremname}
\newcommand{\thistheoremname}{} 

\newcommand{\dda}{\mathord{\mbox{\makebox[0pt][l]{\raisebox{-.4ex}{$\downarrow$}}$\downarrow$}}}

\newcommand{\ua}{\mathord{\uparrow}}
\newcommand{\da}{\mathord{\downarrow}}
\newcommand{\rom}[1]{\rm{\uppercase\expandafter{\romannumeral #1}}}

\newcommand{\fin}{\mathrm{fin}}
\newcommand{\set}[2]{\{#1\mid#2\}}

\newcommand{\pt}[1]{\mathrm{pt}(#1)}

\makeatletter
\def\ps@pprintTitle{%
  \let\@oddhead\@empty
  \let\@evenhead\@empty
  \def\@oddfoot{\reset@font\hfil\thepage\hfil}
  \let\@evenfoot\@oddfoot
}
\makeatother

\begin{document}

\title{Notes on countable frames} \thanks{This research is conducted at TYMC during the Seminar on Domain Theory and Its Application (December 14-20, 2025).}

\author{Xiaodong Jia, Xiaoyong Xi}

\address{X.\,Jia,  School of Mathematics, Hunan University, Changsha, Hunan, 410082, China. Email:  {\rm jiaxiaodong@hnu.edu.cn} }

\address{X.\,Xi, School of Mathematics and Statistics, Yancheng Teachers University, Yancheng 224002, China. Email: {\rm xixy@yctu.edu.cn}}

\begin{abstract}
Matthew de Brecht raised the question of whether countable frames are continuous lattices. We prove that the continuity of a countable frame implies the quasicontinuity of its corresponding spectrum in the dual specialization order. We further show that this question admits a positive answer if the frame's spectrum is a $T_1$ space or a Scott space. In general, we confirm the existence of non-continuous countable frames. This work also partially addresses an open problem proposed by Jimmie Lawson and Michael Mislove in 1990, which concerns the characterization of when the spectrums of spatial frames are Scott spaces.
\end{abstract}

\maketitle

\section{Introduction}

At the Seminar on Domain Theory and Its Application held at Tianyuan Mathematics Research Center (TYMC, December 14-20, 2025), Matthew de~Brecht asked whether countable frames are actually continuous lattices. As every countable frame is spatial~\cite[Lemma~1.1]{erne2019web}, the spectrum, through the Stone duality, of a countable frame is exactly a sober topological space that has only countably many open sets.  In the light of the Hofmann-Mislove Theorem~\cite[Theorem II-1.20]{gierz03}, sober spaces with a continuous open set lattice are precisely locally compact sober spaces~\cite[Theorem 8.3.10]{goubault13a}. Then de~Brecht's question is equivalent to asking whether every sober topological space with only countably many open sets is actually locally compact. In this note we show that such a space is indeed locally compact (even finite) when it is $T_1$, and is a quasialgebraic domain in the sense of Dana Scott's domain theory, when it is a  directed-complete poset (dcpo) equipped with the Scott topology, i.e., when it is a \textit{Scott space}. In the latter case, we show that the countable frame in consideration must be continuous. That  partially answers an open problem posed by Jimmie Lawson and Michael Mislove in~\cite[Problem 528]{lawson-mislove-90}, for characterizing when the spectrums of spatial frames are Scott spaces.

However, in general, we give examples to show that there exist countable frames that are not continuous. Hence de~Brecht's question has a negative answer generally. 

\section{Preliminaries}

\subsection{Order and topology}
We will use the concept of directed-complete partially ordered sets, \textit{dcpo}'s for short. They are posets in which all directed subsets have suprema. Let $P$ be a poset and $x\in P$, we use $\ua x$ to denote 
$\set{y\in P}{x\leq y}$. For $A\subseteq P$, let $\ua A = \bigcup_{x\in A} \ua x$. A subset $A$ of $P$ is called an \textit{upper set} if $A = \ua A$. Dually we can define $\da x$, $\da A$ and \textit{lower sets}. 
In a poset $P$, an upper subset $O$ is said to be \textit{Scott open} if for any directed subset $D$ of $P$, that the supremum $\sup D$ of $D$ exists and $\sup D \in O$ imply $D\cap O \not= \emptyset$. All Scott open subsets of $P$ form the \textit{Scott topology} $\sigma P$ on $P$. A dcpo equipped with the Scott topology is referred as to a \textit{Scott space}. We will also use the \textit{upper topology} on $P$, which is generated, as subbasic closed sets, by sets $\da F$ for finite subsets $F$ of $P$. 

On a topological space $X$, the specialization preorder $\leq$ on $X$ is defined as $x\leq y$ if and only if $x$ is in the closure of $\{y\}$, for $x, y\in X$. The space $X$ is $T_0$ if and only if the preorder $\leq$ is an partial order on $X$. 
In this note, only $T_0$ spaces are considered and orders used on spaces are the specialization orders until stated otherwise. 
An \textit{irreducible} subset $A$ of $X$ is a set such that $A \subseteq B\cup C$ implies that $A\subseteq B$ or $A\subseteq C$ for all closed subsets $B$ and $C$. A \textit{sober} space is a space in which every irreducible closed subset is the closure of some unique singleton. Every sober space is a $T_0$ space, and even a \textit{monotone convergence space} \cite[Exercise O-5.15]{gierz03}, in the sense that the specialization order on $X$ turns $X$ into a dcpo, and the original opens on $X$ are Scott open. Take a closed subset $C$ in a sober space $X$. So $C$ is a Scott closed subset in $X$, and then every element in $C$ is below some element in the set $\max C$ of all maximal elements of $C$. 

\subsection{Domain theory} Dana Scott's domain theory \cite{gierz03, abramsky94} provides a general mathematical framework for denotational semantics for functional programming languages. In a poset $P$, an element $x$ is \textit{way-below} $y$, $x\ll y$ in symbols, if for every directed subset $D \subseteq P$ such that $\sup D$ exists, that $y\leq \sup D$ implies that $x\leq d$ for some $d\in D$. We use $\dda y$ to denote the set $\set{x\in P}{x\ll y}$. A poset $P$ is called a \textit{continuous poset}, if for each $y\in P$, $\dda y$ is directed and $y = \sup \dda y$. A continuous dcpo is often called as a \textit{continuous domain} or simply a \textit{domain}.
A poset $P$ is called \textit{algebraic}, if for each $y\in P$, the set $\set{x\in P}{x\ll x \leq y}$ is directed and has $y$ as its supremum. Algebraic dcpo's are called \textit{algebraic domains}. Obviously, algebraic domains are continuous domains. 
The way-below relation on posets can be generalized to a relation between subsets of posets. 
In a poset $P$, we say a subset $F$ is way-below a subset $G$, in symbols $F\ll G$, if for every directed subset $D \subseteq P$ such that $\sup D$ exists, that $\sup D\in \ua G$ implies that $d\in \ua F$ for some $d\in D$. We say that 
a poset $P$ is a \textit{quasicontinuous poset}, if for each $x\in P$, the set $\fin (x) = \set{F\ll x}{F\subseteq_\fin P}$ of all finite subsets of $P$ that are way-below $x$ is directed in the Smyth preorder $\sqsubseteq$, and $\ua x = \bigcap_{F\in \fin(x)} \ua F$, where $F\sqsubseteq G$ if and only if $G\subseteq \ua F$. quasicontinuous dcpo's are called \textit{quasicontinuous domains}. If in addition, for each $x\in P$, the set $\fin_c (x) = \set{F\subseteq_\fin P}{F\ll F\ll x}$ is directed in the Smyth preorder and $\ua x = \bigcap_{F\in \fin_c(x)}\ua F$, then $P$ is called \textit{quasialgebraic}. Similarly, quasialgebraic dcpo's are called \textit{quasialgebraic domains}. 
One could easily see that algebraic domains are quasialgebraic, and continuous domains are quasicontinuous. Quasicontinuous domains, hence quasialgebraic domains and continuous domains, are sober in the Scott topology. 
Morphisms between dcpo's are always considered to be \textit{Scott-continuous}, that is, monotone maps that preserve suprema of directed subsets. 

\subsection{Stone duality} The celebrated Stone duality sets up an adjunction between the category of $T_0$ topological spaces and the category of frames. When restricted to sober spaces and spatial frames, the adjunction becomes an equivalence of categories. In general, a frame is a complete lattice on which finite infima distribute over arbitrary suprema, and accordingly, morphisms between frames are maps that preserve finite infima and arbitrary suprema. 
A \textit{spatial} frame is a frame in which every element $x$ can be written as the infimum of all \textit{prime} elements that are above $x$. In the Stone duality, every $T_0$ topological space $X$ is sent to the open set lattice $\mathcal O(X)$, which is a spatial frame, and dually  every frame $L$ to the set of all its prime elements equipped with the \textit{hull-kernel} topology. The resulting space is called the \textit{spectrum} of~$L$ and is always sober. The spectrum of~$L$ will be denoted by $\mathrm{pt}(L)$. For a nice  exposition on Stone duality, the reader is referred to \cite[Chapter 8]{goubault13a}. As every countable frame is spatial \cite[Lemma 1.1]{erne2019web}, the spectrums of countable frames are exactly sober topological spaces with  countably many open sets, and conversely, the Stone duals of such spaces are obviously countable frames. 

\subsection{Lawson and Mislove's open question} The Stone duality restricts to many subclasses of sober spaces and the corresponding classes of spatial frames. To illustrate, we have
 \begin{itemize}
     \item the category of domains (equipped with the Scott topology) is in dual with that of completely distributive lattices \cite{lawson79}; 
     \item the category of quasicontinuous domains (equipped with the Scott topology) is in dual with that of distributive \textit{hypercontinuous} lattices \cite[Proposition VII-3.8]{gierz03};
     \item the category of locally compact sober spaces are in dual with that of distributive continuous lattices \cite[Proposition V-5.20]{gierz03}.
 \end{itemize}

In the first two bullets, we notice that the hull-kernel topology on the spectrum coincides with the Scott topology. However in the third bullet, the spectrums are locally compact sober spaces and the (hull-kernel) topology on them may not be the Scott topology arisen from the specialization order. In other words, unlike the other two cases, the spectrums need not be Scott spaces. Jimmie Lawson and Michael Mislove asked in~\cite[Problem 528]{lawson-mislove-90} the question to characterize those distributive continuous lattices for which the spectrum is a Scott space. We will see that if a frame $L$ is countable, without continuity assumed, then among other equivalent properties, $\pt{L}$ is a Scott space is equivalent to that $L$ is continuous. See Theorem~\ref{scott-main-theorem}. 

\section{The general case}

\begin{lemma}\label{lemma1}
    Let $X$ be a sober topological space. If $C$ is a closed subset of $X$ and $\max C$ is infinite, then the open set lattice $\mathcal O(X)$ of $X$ is uncountable. 
\end{lemma}

\begin{proof}
    As sober spaces are monotone convergence spaces, the set $\max C$ of maximal elements of $C$ is nonempty for every nonempty closed subset $C$. 
    We first claim that for every closed subset $C$ of $X$, if $\max C$ is infinite, then there is some $c\in \max C$ such that $c$ has an open neighborhood $U_c$ with the property that $\max C\setminus U_c$ is infinite. Assume that the claim is not true, then we have a closed subset $C$ of $X$ and $\max C$ is infinite such that, for every $c\in \max C$, every open neighborhood $U_c$ of $c$, the set $\max C\setminus U_c$ is finite. 
    Consider any given open subsets $U, V$ such that 
    $U\cap \max C \not= \emptyset$ and  $V\cap \max C \not= \emptyset$. Then we know that both $\max C\setminus U$ and  $\max C \setminus V$ are finite. It follows that their union $\max C\setminus (U\cap V)$ is finite, 
    and then $U\cap V \cap \max C$ is an infinite set, whence, $U\cap V \cap \max C \not= \emptyset$. This implies that $C$ is an irreducible closed set. As $X$ is a sober space, $C = \da \max C$ would have been a closure of a singleton. But $\max C$ is infinite. This is impossible. So the claim holds. 

    Now by the claim, we know there is an $x_1 \in A$ and an open set $U_1$ containing $x_1$ with that $\max C\setminus U_1$ is infinite. Notice that $C\setminus U_1$ is closed and $\max (C\setminus U_1) = \max C\setminus U_1$ is infinite. Applying the claim again to $C\setminus U_1$, we find an $x_2 \in \max C\setminus U_1$ and an open set $U_2$ containing $x_2$ but not $x_1$ (if $U_2$ contains $x_1$, update $U_2$ to $U_2\setminus \da x_1$) such that $\max C \setminus (U_1 \cup U_2)$ is infinite. Repeat this procedure  infinitely many times, and we get a sequence of points $x_1, x_2, \cdots x_i, \cdots, i\in \mathbb N$ and a sequence of open subsets $U_1, U_2, \cdots, U_i, \cdots, i\in \mathbb N$, such that for each $i\in \mathbb N$, $x_i\in U_i$ and $U_i$ does not contain $x_j$, where $j< i$. Note that for $j> i$, $x_j$ is chosen in $\max C\setminus \bigcup_{k\leq j-1} U_i$. So $U_i$ does not contain $x_j$ for $j> i$ neither. That is, $x_i \in U_j$ if and only if $i = j$, for $i, j\in \mathbb N$.

    Now take $I, J\subseteq \mathbb N$, and suppose that $k\in I \setminus J$ for some $k\in \mathbb N$. Then $x_k$ is an element in $\bigcup_{i\in I} U_i$ but not in $\bigcup_{j\in J} U_j$.
    Hence $\bigcup_{i\in I} U_i, I\subseteq \mathbb N$ form a family of open subsets of $X$, and this family is uncountable.
\end{proof}

\begin{lemma}\label{lemma2}
    Let $X$ be sober topological space. If the open set lattice $\mathcal O(X)$ of $X$ is countable, then the topology on $X$ is the upper topology on $X$ with the specialization order. In particular, all closed sets are of the form $\da F$ for some finite subset $F$ of $X$.
\end{lemma}
\begin{proof}
    Obviously, the upper topology is coarser than the topology on $X$. This is because $X$ is sober hence $T_0$, and the upper topology is the coarsest $T_0$ topology that recovers the specialization order on $X$. 

    Conversely let  $C$ be a closed subset of $X$. As $X$ is sober, hence monotone convergence \cite[Exercise O-5.15]{gierz03}, so $C$ is Scott closed in the specialization order. Then we know that $C = \da \max C$. 
    As $\mathcal O(X)$ is countable, by the contraposition of Lemma~\ref{lemma1}, $\max C$ is finite. Obviously, $C = \da \max C$ is closed in the upper topology on $X$. 
\end{proof}

The following two corollaries are straightforward consequences of Lemma~\ref{lemma2}.

\begin{corollary}\label{sobert1isfinite}
    Let $X$ be sober $T_1$ topological space. If the open set lattice $\mathcal O(X)$ of $X$ is countable, then $X$, hence $\mathcal O(X)$, must be finite. 
\end{corollary}
\begin{proof}
    Note that $X$ is closed, and $X = \max X$ as $X$ is $T_1$. So $X$ is finite by Lemma~\ref{lemma2}.
\end{proof}

\begin{corollary}    
    Let $L$ be a countable frame. If $\mathrm{pt}(L)$, the spectrum of $L$ is $T_1$, then $L$ is finite. 
\end{corollary}
\begin{proof}
    As $L$ is a countable frame, it is spatial~\cite[Lemma~1.1]{erne2019web}. Hence the spectrum space  $\mathrm{pt}(L)$ is a sober $T_1$ space, and its open set lattice is isomorphic to $L$, hence that the family of opens of $\mathrm{pt}(L)$ is countable. Then the statement is a corollary to Corollary~\ref{sobert1isfinite}. 
\end{proof}

\begin{theorem}\label{main}
    Let $X$ be sober topological space with a countable  open set lattice $\mathcal O(X)$. Consider the following statements:
\begin{enumerate}
    \item $\mathcal O(X)$ is a continuous domain;
    \item $\mathcal O(X)$ is an algebraic domain;
    \item $(X, \geq)$, where $\geq$ is the dual of the specialization order on $X$, is a quasicontinuous poset;
    \item $(X, \geq)$, where $\geq$ is the dual of the specialization order on $X$, is a quasialgebraic poset.
\end{enumerate}
Then $(1)$ is equivalent to $(2)$, and either $(1)$ or $(2)$ implies $(3)$ and $(4)$. If $(X, \geq)$ is a dcpo, then all the four statements are equivalent. 
\end{theorem}
\begin{proof}
    As $\mathcal O(X)$ is countable, the equivalence between $(1)$ and $(2)$ is shown in \cite{jia-coutabledomains}, where the authors show that countable continuous  domains are actually algebraic domains. 

    For the implication from $(2)$ to $(4)$, 
    we note that $\mathcal O(X)$ is order isomorphic to the set $(\mathcal C(X), \supseteq)$ of closed subsets of $X$ in the reverse inclusion order. By Lemma~\ref{lemma2}, all closed subsets of $X$ are of the form $\da F$,  where $F$ is a finite subset of $X$. Assume that $\mathcal O(X)$ is algebraic. So $(\mathcal C(X), \supseteq)$ is algebraic. 
     One could see that $\da F\ll \da F \ll \da x$ in  $\mathcal C(X)$ implies that $F\ll F\ll x$ in $(X, \geq)$. Hence the algebraicity of $\mathcal O(X)$ implies quasialgebraicity of $(X, \geq)$.

    That $(4)$ implies $(3)$ is obvious, as quasialgebraic domains are always quasicontinuous. 

    Consider now that $(X, \geq)$ is a dcpo.
    We finish the proof by showing the equivalence between $(1)$ and $(3)$.
    By \cite[Proposition 4.5]{heckmann13}, a dcpo $P$ is quasicontinuous if and only if the family of all $\ua F$, where $F$ are finite subsets of $P$, is continuous in the reverse inclusion order\footnote{That $P$ is a dcpo is an indispensable condition.}. Applying this result to $(X, \geq)$, we have that $(X, \geq)$ is quasicontinuous if and only if all $\da F$, where $F$ are finite subsets of $X$, form a continuous domain in the reverse inclusion order. Lemma~\ref{lemma2} tells us that all such $\da F$'s are precisely all the closed subsets of $X$. So $(X, \geq)$ is quasicontinuous if and only if $\mathcal O(X)$ is continuous. 
\end{proof}

In the light of Theorem~\ref{main}, we now present a countable frame that fails to be a continuous lattice.

\begin{example}\label{example-counter}
    \begin{figure}[t]
    \centering
        \begin{tikzpicture}[
        dot/.style={circle,fill,inner sep=1.2pt},
        >=stealth, ]
        \node[dot,label=above:{$-1$}] (top) at (0,4) {};
        \node[dot,label=below:{$-\infty$}] (bot) at (0,-4) {};
        \def\xs{-2,-1,0,1}
        \def\ys{2,1,0,-1,-2,-3}
    
        \foreach \x in \xs {
            \foreach \y in {2,1,0,-1} {
                \node[dot] at (\x,\y) {};
            }
            \draw[dashed] (\x,-1) -- (\x,-3);
            \draw (top) -- (\x,2);
            \draw (\x,-1) -- (\x,2);
        }

        \fill[gray!30, opacity=0.5] (0, 0.45) ellipse (0.4 and 3.8);
        
        \foreach \y in \ys {
            \draw[dashed] (2,\y) -- (4,\y);
        }
        \draw[dashed] (top) -- (3,2);
        \end{tikzpicture}
    \caption{$P$}
    \label{fig:example}
    \end{figure}
    Consider the complete lattice $P$ obtained by taking infinitely many disjoint copies of the extended negative numbers with $-\infty$ (with the usual order) and gluing them together at $-\infty$ and $-1$. The poset $P$ is depicted in Picture~\ref{fig:example}.  
    For the copies of negative numbers (without $-\infty$), we call them columns of $P$. For example, in Picture~\ref{fig:example} elements in the gray area form a column of $P$. 
    We give the upper topology $\nu P$ to $P$. One could verify that the closed subsets of $P$ are of the form $\da F$, where each $F$ is finite subset of $P$. Then we know that for a proper closed subset of the form $\da F$, $F$ only ranges on finitely many columns. So all opens on $P$ form a countable frame. Additionally, as $P$ is a complete lattice, $P$ is sober in the upper topology~\cite[Proposition~1.7]{schalk93a}. However, the order dual of $P$ is a typical complete lattice that fails to be quasicontinuous. By Theorem~\ref{main}, the open set lattice $\mathcal \nu P$ cannot be a continuous lattice.  
\end{example}

\begin{remark}
    If we remove the least element $-\infty$ from the complete lattice~$P$ in Example~\ref{example-counter}, then the resulting poset $Q$ is a complete sup-semilattice. Again by~\cite[Proposition~1.7]{schalk93a}, $Q$ is sober in the upper topology $\nu Q$ and it has only countably many opens. 
    We verify that $(Q, \nu Q)$ fails to be locally compact neither. 
    Suppose that $K$ is a compact saturated set of $Q$, and $K$ contains a nonempty open subset $O = Q \setminus \da G$, where $G$ is finite. Then there is a whole column $I$ of $Q$ such that $I$ is contained in $K$ (take any column $I$ that does not insect the finite set $G$). Note that $K$ is compact, the column $I$ must be above some minimal element in $K$. However, no elements in $Q$ are below all elements in $I$. So we reach a contradiction. It means that every compact saturated subset of $Q$ has empty interior. Thus, $P$ is not locally compact.
    Note that the dual poset of $Q$ is a continuous poset. This shows that in general one cannot expect an equivalence between  Statements $(1)$ and $(3)$ in Theorem~\ref{main}.
\end{remark}

\section{A partial answer to Lawson and Mislove's question}

Now we consider countable frames that arise as Scott open set lattices of dcpo's, and we will see that such frames are actually continuous, hence algebraic, and they are even hypercontinuous lattices.  
We use $\Sigma L$ to denote the resulting Scott space by endowing the Scott topology $\sigma L$ on the dcpo $L$. 
We will need the following result, which can be found as \cite[Theorem 2.5.7]{jia2018meet}.

\begin{lemma}\label{jia18thesis}
    Let L be a dcpo. Then the following statements are equivalent:
  \begin{enumerate}
      \item $L$ is core-compact, i.e., $\sigma L$ is a continuous lattice in the inclusion order;
      \item for every dcpo $S$, one has $\Sigma (S\times L) = \Sigma S \times \Sigma L$ \footnote{Note that these two products are taken in the category of dcpos and that of topological spaces, respectively.};
      \item For every complete lattice $S$, one has $\Sigma (S\times L) = \Sigma S \times \Sigma L$;
      \item $\Sigma (\sigma L\times L) = \Sigma \sigma L \times \Sigma L$.
  \end{enumerate}
    
\end{lemma}

\begin{proposition}\label{prop-scott-countable}
    Let $P$ be a sober Scott space with only countably many Scott open sets. Then $\sigma P$ is an algebraic lattice. 
\end{proposition}

\begin{proof}
    Since $P$ and $\sigma P$ are countable dcpo's, by \cite[Lemma 3.1]{miao2024productssoberdcposneed}, we know that the Scott topology on $\sigma P \times \sigma \sigma P$ is the same as the product topology on $\Sigma \sigma P \times \Sigma \sigma \sigma P$. The equivalence between $(1)$ and $(4)$ in Lemma~\ref{jia18thesis} implies that $\sigma P$ is core-compact. That is, $\sigma \sigma P$ is a continuous lattice in the inclusion order. By the equivalence between $(1)$ and $(2)$ in Lemma~\ref{jia18thesis}, we know that the Scott topology on $P\times \sigma P$ is then equal to the product topology on $\Sigma P \times \Sigma \sigma P$. Now applying the equivalence between $(1)$ and $(4)$ in Lemma~\ref{jia18thesis} again, we know that $P$ is core-compact. So $\sigma P$ is a countable continuous lattice, and hence it is an algebraic lattice by results in~\cite{jia-coutabledomains}. 
\end{proof}

As promised, the following theorem provides a partial answer to Lawson and Mislove’s question  for characterizing those frames for which the spectrum is a Scott space. The general question still remains open. 

\begin{theorem}\label{scott-main-theorem}
    Let $L$ be a countable frame and $\mathrm{pt}(L)$ be its spectrum. Then the following are equivalent. 
    \begin{enumerate}
        \item $\mathrm{pt}(L)$ is a Scott space;
        \item $\mathrm{pt}(L)$ is a locally compact sober Scott space;
        \item $\mathrm{pt}(L)$ is quasicontinuous;
        \item $\pt{L}$ is quasialgebraic;
        \item $L$ is quasicontinuous;
        \item $L$ is quasialgebraic;
        \item $L$ is continuous;
        \item $L$ is algebraic;
        \item $L$ is hypercontinuous. 
    \end{enumerate}
\end{theorem}
\begin{proof}
    We prove the following sequence of implications: $$(1)\Rightarrow (8) \Leftrightarrow (7) \Leftrightarrow (6) \Leftrightarrow (5) \Rightarrow (4) \Rightarrow (3) \Leftrightarrow (9) \Rightarrow (2) \Rightarrow (1).$$  
    
\begin{itemize}[itemindent=1.2cm]
    \item[$(1)\Rightarrow (8).$] 
As $\sigma (\pt{L})$ is isomorphic to $L$, hence countable, so $\pt{L}$ is a Scott space with only countably many Scott open sets. That $(1)\Rightarrow (8)$ follows directly from Proposition~\ref{prop-scott-countable}.
    \item[$(5)\Rightarrow (8).$] One can easily see that $(8) \Rightarrow (7) \Rightarrow (5)$ and  $(8) \Rightarrow (6) \Rightarrow (5)$. We prove that $(5)\Rightarrow (8)$. For that, we realize that the frame $L$ is a meet-continuous lattice (binary infima distribute over directed suprema), and quasicontinuous and meet-continuous dcpo's are continuous domains \cite[Proposition III-3.10]{gierz03}. So $L$ is continuous. As $L$ also is countable, it is algebraic by results in~\cite{jia-coutabledomains}. 
    \item[$(8)\Rightarrow (4).$] If $L$ is algebraic, then the Scott open set lattice $\sigma (\pt{L})$ of $\pt{L}$ is algebraic.  By the result in \cite[Proposition~7(3)]{ERNE20092054}, we know that every compact element in $\sigma(\pt{L})$ must be of the form $\ua F$ for some finite subset $F$ of $\pt{L}$, and they form a basis for $\Sigma\pt{L}$. This proves that $\pt{L}$ is quasialgebraic. 
\end{itemize} 
That $(4) \Rightarrow (3) \Rightarrow  (2) \Rightarrow (1)$ is obvious. Finally, the equivalence between $(3)$ and $(9)$ is a well-known result about restricting the Stone duality to the class of quasicontinuous domains and that of hypercontinuous lattices~\cite[Proposition VII-3.8]{gierz03}. 
\end{proof}

\section*{Acknowledgments}
The authors thank Matthew de Brecht for asking the related questions. They also acknowledge the Tianyuan Mathematics Research Center for hosting the Seminar on Domain Theory and Its Applications in December 2025, which the authors found quite a simulating and productive event. The first author thanks Dr.\,Zhenchao Lyu for pointing to \cite[Proposition~7]{ERNE20092054}. The authors are supported by NSF of China (Grants:12371457, 12571507).

\bibliographystyle{plain}

\end{document}